\documentclass[final,1p,times]{elsarticle}
\usepackage{fancyhdr}
\usepackage{amssymb,amsmath,graphicx,amssymb,amsfonts,wrapfig,color}
 \usepackage{amsthm,setspace}

\newtheorem{theorem}{Theorem}[section] % reset theorem numbering for each section
\newtheorem{lemma}{Lemma}[section] 
 
\newtheorem{example}{Example}[section] 
\newtheorem{corollary}{Corollary}[section] 
\newtheorem{remark}{Remark}[section]

\journal{Discrete Math}

\begin{document}

\begin{frontmatter}

%% Title, authors and addresses

%% use the tnoteref command within \title for footnotes;
%% use the tnotetext command for theassociated footnote;
%% use the fnref command within \author or \address for footnotes;
%% use the fntext command for theassociated footnote;
%% use the corref command within \author for corresponding author footnotes;
%% use the cortext command for theassociated footnote;
%% use the ead command for the email address,
%% and the form \ead[url] for the home page:
%% \title{Title\tnoteref{label1}}
%% \tnotetext[label1]{}
%% \author{Name\corref{cor1}\fnref{label2}}
%% \ead{email address}
%% \ead[url]{home page}
%% \fntext[label2]{}
%% \cortext[cor1]{}
%% \address{Address\fnref{label3}}
%% \fntext[label3]{}

\title{Constructions for orthogonal designs using signed group orthogonal designs}

%% use optional labels to link authors explicitly to addresses:
%% \author[label1,label2]{}
%% \address[label1]{}
%% \address[label2]{}

\author{Ebrahim Ghaderpour}
\address{ Department of Mathematics and Statistics, University of Calgary, Calgary, AB, Canada T2N 1N4 \\
              Email: {ebrahim.ghaderpour@ucalgary.ca}}
%\address[P]{ Lassonde School of Engineering, York University, Canada. \\
              %Email: {spiros@yorku.ca}}

%\address{}

\begin{abstract}
%% Text of abstract
Craigen introduced and studied signed group Hadamard matrices extensively and eventually provided an asymptotic existence result for Hadamard matrices. Following his lead, Ghaderpour introduced signed group orthogonal designs and showed an asymptotic existence result for orthogonal designs and consequently Hadamard matrices. In this paper, we construct some interesting families of orthogonal designs using signed group orthogonal designs to show the capability of signed group orthogonal designs in generation of different types of orthogonal designs.

% \PACS{PACS code1 \and PACS code2 \and more}
% \subclass{MSC code1 \and MSC code2 \and more}

\end{abstract}

\begin{keyword}
%% keywords here, in the form: keyword \sep keyword
Circulant matrix\sep Golay pair \sep Hadamard matrix \sep Orthogonal design  \sep Signed group orthogonal design. % \sep Spectrum.

%% PACS codes here, in the form: \PACS code \sep code

%% MSC codes here, in the form: \MSC code \sep code
%% or \MSC[2008] code \sep code (2000 is the default)

\end{keyword}

\end{frontmatter}

%% \linenumbers
%\doublespacing
%% main text
\sloppy
\section{Preliminaries}
 A {\it Hadamard} matrix \cite{GS, Wolfe} is a square matrix with entries from $\{\pm 1\}$ whose rows are pairwise orthogonal.
An {\it orthogonal design} (OD) \cite{Craigenthesis, GS, Wolfe} of order $n$ and type $(c_1, \ldots, c_k)$, denoted by $OD(n; \ c_1, \ldots, c_k)$, 
 is a square matrix $X$ of order $n$ with entries from $\{0, \pm  x_1, \ldots, \pm  x_k\}$  that satisfies
\begin{align*}
XX^{\rm T}= \Big(\sum_{j=1}^k c_j x_j^2 \Big)I_n,
\end{align*}
where  the $c_j$'s are positive integers, the $x_j$'s are commuting variables, $I_n$ is the identity matrix of order $n$, and $X^{\rm T}$ is the transpose of $X$.
An OD with no zero entry is called a {\it full} OD. A Hadamard matrix can be obtained by equating all variables of a full OD to 1. The maximum number of variables in an OD of order $n=2^ab,$ $b$ odd, is $\rho(n)=8c+2^d$, where $a=4c+d$, $0\le d<4$. This number is called {\it Radon-Hurwitz number} \cite[Chapter 1]{GS}.

A {\it complex orthogonal design} (COD) \cite{Craigenthesis, Geramita, Ghaderpour} of order $n$ and type $(c_1, \ldots, c_k)$, denoted by $COD(n; \ c_1, \ldots, c_k)$, 
 is a square matrix $X$ of order $n$ with entries from $\{0, \pm  x_1, \pm i x_1, \ldots, \pm  x_k, \pm i x_k\}$  that satisfies
\begin{align*}
XX^*= \Big(\sum_{j=1}^k c_j x_j^2 \Big)I_n,
\end{align*}
where the $c_j$'s are positive integers, the $x_j$'s are commuting variables, and $*$ is the conjugate transpose.

Two matrices $A$ and $B$ of the same dimension are called {\it disjoint} \cite{GS, EG, Wolfe} if the matrix computed via entrywise multiplication of $A$ and $B$ is a zero matrix.
Pairwise disjoint matrices such that their sum has no zero entries are called {\it supplementary} \cite{Craigenthesis, GS}.

The {\it Kronecker product} \cite{GS, EG} of two matrices $A = [a_{ij}]$ and $B $ of orders  
 $m\times n$ and $r\times s,$ respectively, denoted by $A \otimes B$, is defined by 
$$
 A \otimes B := \begin{bmatrix}
                a_{11}B & a_{12}B & \cdots & a_{1n}B\\
                a_{21}B & a_{22}B & \cdots & a_{2n}B\\
                \vdots  &         &        & \vdots\\
                a_{m1}B & a_{m2}B & \cdots & a_{mn}B\\
               \end{bmatrix}\!,
 $$ 
that is a matrix of order $mr \times ns$.

The {\it non-periodic autocorrelation function} \cite{KHH}  of a sequence $A=(x_1, \ldots, x_n)$ of commuting square complex matrices of order $m,$ is defined by
$$N_{A}(j):=\left\{
      \begin{array}{l l}
      \displaystyle \sum_{i=1}^{n-j} x_{i+j}x_i^* \ \ \ {\rm if} \ \  j=0,1,2,\ldots, n-1,& \\
       0 \ \ \ \ \ \ \ \ \ \ \ \ \ \ j\ge n,& 
       \end{array} \right.$$
where $*$ is the conjugate transpose. 
A set $\{A_1, A_2, \ldots, A_{\ell}\}$ of sequences (not necessarily in the same length) is said to have zero autocorrelation if for all $j>0,$
$\sum_{k=1}^{\ell}N_{A_k}(j)=0.$  Sequences having zero autocorrelation are called {\it complementary} \cite{GS}.

A pair $(A;B)$ of $\{\pm 1\}$-complementary sequences of length $n$ is called a {\it Golay pair}  of length $n$. A {\it Golay number} is a positive integer $n$ such that there exists a Golay pair of length $n$. 
Similarly, a pair $(C;D)$ of $\{\pm 1, \pm i\}$-complementary sequences  of length $m$ is called a {\it complex Golay pair} of length $m$. A {\it complex Golay number} is a positive integer $m$ such that there exists a complex Golay pair of length $m$ \cite{CHK, Craigen, GS}.

A {\it signed group} $S$ \cite{Craigen, Craigenthesis, Ghaderpour}  is a group with a distinguished central element of order two. We denote the unit of a signed group by 1 and the distinguished central element of order two by $-1$.  In every signed group, the set $\{1,-1\}$ is a normal subgroup, and the {\it order} of signed group $S$ is the number of elements in the quotient group $S/\big<-1\big>$. Therefore, a signed group of order $n$ is a group of order $2n.$ 

For instance, the trivial signed group $S_{\mathbb R}=\{1,-1\}$ is a signed group of order $1$, the complex signed group $S_{\mathbb C}=\big<i\colon\,  i^2=-1\big>=\{\pm 1, \pm i\}$ is a signed group of order $2$, the quaternion signed group $S_Q=\big<j,k\colon\,  j^2=k^2=-1, jk=-kj\big>=\big\{\pm 1, \pm j, \pm k, \pm jk\big\}$ is a signed group of order $4$, and the set of all monomial $\{0, \pm 1\}$-matrices of order $n$, $SP_n$, forms a group of order $2^nn!$ and a signed group of order $2^{n-1}n!.$ The distinguished central elements of $S_{\mathbb R}$, $S_{\mathbb C}$ and $S_Q$ are all $-1$, and the distinguished central element of $SP_n$ is $-I_n$, where $I_n$ is the identity matrix of order $n$.

A signed group $S'$ is called a {\it signed subgroup} \cite{Craigen, Craigenthesis} of a signed group $S$, denoted by $S'\le S$, if $S'$ is a subgroup of $S$, and the distinguished central elements of $S'$ and $S$ coincide.  As an example, we have $S_{\mathbb R}\le S_{\mathbb C}\le S_Q$.

Let $S$ be a signed group and $T\le SP_n$. A {\it remrep}  ({\it real monomial representation}) \cite{CK, Craigen, Craigenthesis}  of degree $n$ is a map $\phi: S \rightarrow T$ such that for all $a,b \in S$, $\phi(ab)=\phi(a)\phi(b)$ and $\phi(-1)=-I_n$.

If $R$ is a ring with unit $1_R$, and $S$ is a signed group with distinguished central element $-1_S$,  then  
$R[S]:=\big\{\sum_{i=1}^n s_i r_i\colon\,  s_i\in \varrho,  \ r_i\in R\big\}$ is a signed group ring \cite{Craigen, Craigenthesis}, where $\varrho$ is a set of coset representatives of $S$ modulo $\big<-1_S\big>$. The set $\varrho$ is often referred to as a {\it transversal} of $\big<-1_S\big>$ in $S$.
For $s\in \varrho, r\in R$, we make the identification $-sr=s(-r)$. Addition is defined termwise, and multiplication is defined by linear extension. For instance,  $s_1r_1(s_2r_2+s_3r_3)=s_1s_2r_1r_2+s_1s_3r_1r_3$, where $s_i\in \varrho$, $r_i\in R$ $i\in \{1,2,3\}$.

In this work, we choose $R=\mathbb{R}$. 
If $x\in\mathbb{R}[S]$, then $x=\sum_{i=1}^n s_i r_i$, where $s_i\in \varrho,  r_i\in \mathbb{R}$, and we define the {\it conjugate} of $x$ by 
$\overline{x}:=\sum_{i=1}^n \overline{s_i} r_i=\sum_{i=1}^n s^{-1}_i r_i$.
Clearly, the conjugate is an involution that is $\overline{\overline x}=x$ for all $x\in \mathbb{R}[S]$, and $\overline{xy}=\bar{y}\bar{x}$ for all $x,y\in \mathbb{R}[S]$. As some examples, for any $a,b \in \mathbb{R}$, we have $\overline {a+ib}=a+\bar{i}b=a+i^{-1}b=a-ib$, where $i \in S_{\mathbb C}$, and
$\overline{ja+jkb}=j^{-1}a+(jk)^{-1}b=-ja-jkb$, where $j,k\in S_Q$. 

A {\it circulant} matrix $C$ \cite{Craigenthesis, GS, EG} is a square matrix whose each row vector is rotated one element to the right with respect to the previous row vector, and we denote it by ${\rm circ} \ \big(a_1, a_2, \ldots, a_n\big)$, where $\big(a_1, a_2, \ldots, a_n\big)$ is its first row. The circulant matrix $C$ can be written as $C=a_1I_n+\sum_{k=1}^{n-1}a_{k+1}U^k,$ where $U={\rm circ} \ \big(0,1,0,\ldots,0\big)$ (see \cite[Chapter 4]{GS}). Therefore, any two circulant matrices of order $n$ with commuting entries commute. 
If $C={\rm circ} \ \big(a_1, a_2, \ldots, a_n\big),$ then $C^*={\rm circ} \ \big(\overline{a}_1, \overline{a}_n, \ldots, \overline{a}_2\big)$, where $*$ is the conjugate transpose.

Suppose that $A=\big(a_1, a_2, \ldots, a_n\big)$ is a sequence whose nonzero entries
are elements of a signed group $S$ multiplied on the right by variables $x_i$'s ($1\le i\le k$). We use  $A_{\overline{R}}$ to denote a sequence whose elements are those of $A$, conjugated and in reverse order \cite{CHK, Ghader} that is
$A_{\overline{R}}=\big(\overline{a}_n, \ldots, \overline{a}_2, \overline{a}_1\big)$.

A {\it signed group weighing matrix} (SW) \cite{CK, Craigen} of order $n$ and weight $w$ {\it over} a signed group $S$, denoted by $SW(n,w,S)$, is a $(0,S)$-matrix  (that is a matrix whose nonzero entries are in $S$) $W$ such that $WW^*=wI_n$, where $*$ is the conjugate transpose. An SW over $S$ with no zero entry ($w=n$) is called a {\it signed group Hadamard matrix} (SH) over $S$ \cite{Craigen, Craigenthesis}, denoted by $SH(n,S)$. Note that the matrix operations for SWs and SHs are in the signed group ring $\mathbb{Z}[S]$.

Two square matrices $A$ and $B$ are called {\it amicable} if $AB^*=BA^*$, and they are called {\it anti-amicable} if $AB^*=-BA^*$, where $*$ is the conjugate transpose \cite{Craigenthesis, GS, Ghaderpour, Wolfe}. If the entries of $A$ and $B$ belong to a signed group ring, then the matrix operations are in the signed group ring as mentioned above.

\section{Some non-existence results for signed group orthogonal designs}
A {\it signed group orthogonal design} (SOD) of order $n$ and type $\big(u_1, \ldots, u_k\big)$ {\it over} a signed group $S$, denoted by $SOD\ \big(n; \ u_1, \ldots, u_k, S\big)$, is a square matrix $X$ of order $n$ whose nonzero entries
are elements of $S$ multiplied on the right by commuting variables $x_i$'s ($1\le i\le k$) such that
\begin{align*}
XX^{*}= \Bigg(\sum_{i=1}^k u_i x_i^2 \Bigg)I_n,
\end{align*} 
where $u_1, \ldots, u_k$ are positive integers, and $*$ is the conjugate transpose.  Note that the conjugate of entry $\epsilon x_i$ ($\epsilon \in S$) is $\overline{\epsilon} x_i=\epsilon^{-1} x_i$, and the matrix operations for SODs in this work are in the signed group ring $\mathbb{R}[S]$. 
It is shown \cite{Ghader,Ghaderpour} that if $X$ is an SOD over a finite signed group, then $XX^*=X^*X$. We call an SOD with no zero entries a {\it full} SOD.

\begin{remark}{\rm
In the definition of SOD in \cite{Ghader, Ghaderpour}, the author says that the entries of an SOD are from $\{0, \epsilon_1x_1, \ldots, \epsilon_kx_k\}$ ($\epsilon_i\in S$). What the author means by this arrangement is that each variable may appear in the SOD with various signed group elements as coefficients. In this paper, we also use the notation $SOD\ \big(n; \ u_1, \ldots, u_k, S\big)$ instead of $SOD\ \big(n; \ u_1, \ldots, u_k\big)$ over a signed group $S$.
}\end{remark}

\begin{example}
{\rm Consider the following square matrix: 
\begin{align*}
X= \left[ \begin {array}{rrrr} jkx_1&jx_2&kx_3&x_3\\ jx_2&jkx_1&x_3&kx_3 \\ kx_3&x_3&jkx_1&jx_2 \\ x_3&kx_3& jx_2&jkx_1\end {array} \right]\!,
\end{align*} 
where $x_1, x_2, x_3$ are commuting variables and $j,k\in S_Q$. We have 
\begin{align*}
XX^*=X\overline{X}^{\rm T}= \left[ \begin {array}{rrrr} jkx_1&jx_2&kx_3&x_3\\ jx_2&jkx_1&x_3&kx_3 \\ kx_3&x_3&jkx_1&jx_2 \\ x_3&kx_3& jx_2&jkx_1\end {array} \right]\!
 \left[ \begin {array}{rrrr} \overline{jk}x_1&\bar{j}x_2&\bar{k}x_3&x_3\\ \bar{j}x_2&\overline{jk}x_1&x_3&\bar{k}x_3 \\ \bar{k}x_3&x_3&\overline{jk}x_1&\bar{j}x_2 \\ x_3&\bar{k}x_3& \bar{j}x_2&\overline{jk}x_1\end {array} \right]\!.
\end{align*}
Let $\omega_{a,b}$ be the entry of the row $a$ and column $b$ of matrix $XX^*$. We have
\begin{align*}
\omega_{1,1}&= (jkx_1)(\overline{jk}x_1)+(jx_2)(\bar{j}x_2)+(kx_3)(\bar{k}x_3)+(x_3)(x_3) \\ &= jk\overline{jk} x_1x_1+j\bar{j}x_2x_2+k\bar{k}x_3x_3+x_3x_3 \\& = x_1^2+x_2^2+2x_3^2
\end{align*}
Similarly, it can be verified that $\omega_{2,2}=\omega_{3,3}=\omega_{4,4}=x_1^2+x_2^2+2x_3^2$. Moreover,
\begin{align*}
\omega_{1,3}&=(jkx_1)(\bar{k}x_3)+(jx_2)(x_3)+(kx_3)(\overline{jk}x_1)+(x_3)(\bar{j}x_2) \\ &=jk\bar{k}x_1x_3+jx_2x_3+k\overline{jk}x_3x_1+\bar{j}x_3x_2 \\&= j x_1x_3+jx_2x_3-jx_1x_3-jx_2x_3 
 \ \ \ \ \ \ {\rm commuting \ variables}\\&= j(x_1x_3-x_1x_3)+j(x_2x_3-x_2x_3) \\&=0.
\end{align*}
It can be verified that $\omega_{a,b}=0$ for $1\le a\neq b\le 4$. Therefore,  $XX^*=(x_1^2+x_2^2+2x_3^2)I_4$, and so $X$ is $SOD(4; 1,1,2,S_Q)$.
}
\end{example}

\begin{remark}
{\rm 
Equating all variables to 1 in any SOD results in an SW.  
Equating all variables to 1 in any full SOD results in an SH.
An SOD over the trivial signed group $S_{\mathbb{R}}$ is an OD, and
an SOD over the complex signed group $S_{\mathbb{C}}$ is a COD.}
\end{remark}
The following lemma is immediate from the definition of SOD.
\begin{lemma}{\rm \cite[Chapter 6]{Ghaderpour}}\label{equivalence1}
If $A$ is an  SOD over a signed group $S$, then 
permutations of the rows or columns  of $A$ do not affect  the orthogonality of $A$, and
multiplication of each row of $A$ from the left or each column of $A$ from the right by an element in $S$  does not affect  the orthogonality of $A$.
\end{lemma}
%%%%%%%%%%%%%%%%%%%%%%%%%%%%%%%%%%%%%%%%%%%%%%%%%%%%%%%%%%%
We now show some non-existence results for SODs.
The following lemma is shown in \cite{Ghader}, and for the sake of completeness, we give a proof.
\begin{theorem}\label{musteven}
There does not exist any full SOD of order $n$ over any signed group, if $n$ is odd and $n>1$. 
\end{theorem}
\begin{proof}{
Assume that there is a full SOD of order $n>1$ over a signed group $S$.  Equating all variables to 1 in the SOD, one obtains a $SH(n,S)=[h_{ij}]_{i,j=1}^n.$ From the second part of Lemma \ref{equivalence1}, one may multiply each column of the $SH(n,S)$ from the right by the inverse of corresponding entry of its first row, $\overline{h}_{1j}$, to get an equivalent $SH(n,S)$ with the first row all 1 (see  \cite{CHK, Craigen} for the definition of equivalence). By orthogonality of the rows of the $SH(n,S)$, the number of occurrences of a given element $s\in S$ in each subsequent row must be equal to the number of occurrences of $-s$. Therefore, $n$ has to be even.}
\end{proof}

\begin{lemma}\label{SW6}
There does not exist any $SW(6,3,S)$.
\end{lemma}
\begin{proof}
{Assume that $A$ is $SW(6,3,S)$.  From Lemma \ref{equivalence1}, one may permute the rows and columns of $A$ to obtain a matrix of the following form:
$$A_1=\left[ \begin {array}{cccccc} \bigstar&\bigstar&\bigstar&0&0&0\\ \bigstar& & & & &\\
 \bigstar&&&&&\\ 0&&&&& \\ 0&&&&& \\ 0 &&&&&\end {array} \right]\!,$$
where the $\bigstar$'s are elements in $S$.
Using orthogonality of the first and second rows of $A_1$, only one of the entries at second row and second column or at second row and third column must be zero. Similarly, since the first and second columns of $A_1$ are orthogonal, one of the entries at second row and second column or at third row and second column must be zero. From the first part of Lemma \ref{equivalence1}, since $A_1$ is an SW, the following matrix must be also an SW.
$$A_2=\left[ \begin {array}{cccccc} \bigstar&\bigstar&\bigstar&0&0&0\\ \bigstar&\bigstar &0 &\bigstar & 0&0\\
 \bigstar&0&&&&\\ 0&\bigstar&&&& \\ 0&0&&&& \\ 0 &0&&&&\end {array} \right]\!.$$
Now, orthogonality of the first and second rows with the third row of $A_2$ forces $A_2$ to be of the following form:

$$A_2=\left[ \begin {array}{cccccc} \bigstar&\bigstar&\bigstar&0&0&0\\ \bigstar&\bigstar &0 &\bigstar & 0&0\\
 \bigstar&0&\bigstar&\bigstar&0&0\\ 0&\bigstar&&&& \\ 0&0&&&& \\ 0 &0&&&&\end {array} \right]\!,$$
which contradicts orthogonality of the fifth and sixth columns of $A_2$, so there is no $SW(6,3,S)$. 
}
\end{proof}
\begin{theorem}\label{3tuple}
There exists no $SOD\ (6; \ 3,3, S)$ and no $SOD\ (6; \ 2,2,2, S)$.
\end{theorem}
\begin{proof}{
If there is  $SOD\ (6; \ 3,3, S)$, then equating one of its variables to $0$ and the other one to $1$ results in $SW(6,3,S)$ which contradicts Lemma \ref{SW6}, so there is no  $SOD\ (6; \ 3,3, S)$.

Now suppose that $B$ is $SOD\ (6; \ 2,2,2, S)$. By Lemma \ref{equivalence1}, if one permutes the rows and columns of $B$, then one of the following forms obtains:
$${\footnotesize\left[ \begin {array}{cccccc} \epsilon_{11}a&\epsilon_{12}a&\epsilon_{13}b&\epsilon_{14}b&\epsilon_{15}c& \epsilon_{16}c\\ \epsilon_{21}a&\epsilon_{22}a&\epsilon_{23}b&\epsilon_{24}b&\epsilon_{25}c&\epsilon_{26}c\\
 \epsilon_{31}b&\epsilon_{32}b&\epsilon_{33}c&\epsilon_{34}c&\epsilon_{35}a&\epsilon_{36}a\\
 \epsilon_{41}b&\epsilon_{42}b&\epsilon_{43}c&\epsilon_{44}c&\epsilon_{45}a&\epsilon_{46}a\\
 \epsilon_{51}c&\epsilon_{52}c&\epsilon_{53}a&\epsilon_{54}a&\epsilon_{55}b&\epsilon_{56}b\\
\epsilon_{61}c&\epsilon_{62}c&\epsilon_{63}a&\epsilon_{64}a&\epsilon_{65}b&\epsilon_{66}b\end {array} \right] \ {\rm or} \  
\left[ \begin {array}{cccccc} \gamma_{11}a&\gamma_{12}a&\gamma_{13}b&\gamma_{14}b&\gamma_{15}c& \gamma_{16}c\\ \gamma_{21}a&\gamma_{22}a&\gamma_{23}c&\gamma_{24}c&\gamma_{25}b&\gamma_{26}b\\
 \gamma_{31}b&\gamma_{32}b&\gamma_{33}a&\gamma_{34}a&\gamma_{35}c&\gamma_{36}c\\
 \gamma_{41}b&\gamma_{42}b&\gamma_{43}c&\gamma_{44}c&\gamma_{45}a&\gamma_{46}a\\
 \gamma_{51}c&\gamma_{52}c&\gamma_{53}a&\gamma_{54}a&\gamma_{55}b&\gamma_{56}b\\
\gamma_{61}c&\gamma_{62}c&\gamma_{63}b&\gamma_{64}b&\gamma_{65}a&\gamma_{66}a\end {array} \right]\!,}
$$
where $\epsilon_{ij}, \gamma_{ij}\in S$ ($1\le i,j\le 6$), and $a,b,c$ are commuting variables. 

For the left matrix, consider $\epsilon_{11}=\epsilon_{12}=1,$ so as  in the proof of Theorem \ref{musteven}, orthogonality of the first row with the second and third rows forces $\epsilon_{21}$ to be $-\epsilon_{22}$ and $\epsilon_{31}$ to be $-\epsilon_{32}.$ Thus,
the second and third rows will not be orthogonal, which is a contradiction. 

For the right matrix, consider $\gamma_{21}=\gamma_{22}=1,$ so as in the proof of Theorem \ref{musteven}, orthogonality of the second row with the third and sixth rows forces $\gamma_{31}$ to be $-\gamma_{32}$ and $\gamma_{61}$ to be $-\gamma_{62}.$ Thus,
the third and sixth rows will not be orthogonal, which is a contradiction.
Hence, there is no $SOD\ (6; \ 2,2,2,S)$.}
\end{proof}

From Theorems \ref{musteven} and \ref{3tuple}, there is no $SOD\ \big(1 \cdot 3; \ 1,1,1, S\big)$, $SOD \ \big(2\cdot 3; \ 2,2,2, S\big)$ and $SOD \ \big(3 \cdot 3; \ 3,3,3, S\big)$. However, it is shown \cite{Ghader, Ghaderpour} that there exists $SOD\ \big(4 \cdot 3; \ 4,4,4, S\big)$ and more generally for any $k$-tuple $\big(u_1, \ldots, u_k\big)$ of positive integers, there exists $SOD\ \big(4u; \ 4u_1, \ldots,4u_k, S\big)$, where $u=\sum_{i=1}^k u_i$. An asymptotic existence result for full ODs is obtained in \cite{Ghader, Ghaderpour} by applying Theorem \ref{sodtood} to these full SODs, which in turn improved the  asymptotic existence result for Hadamard matrices obtained by Seberry \cite[Chapter 7]{GS} and Craigen \cite{Craigen}.

\section{Constructions for ODs using SODs of order $2^n$}
In this section, we construct SODs of order $2^n$ with $2^n$ variables over some signed groups, and by using some well-known theorems, we obtain a family of full CODs which, as we shall show, implies the existence of a family of full ODs. We use the notation $u_{(k)}$  to show $u$ repeats $k$ times. 
\begin{theorem}{\rm \cite[Chapter 1]{GS}.}\label{setanti}
For each positive integer $m$, there is a set $$A=\Big\{I_m, A_1, A_2, \ldots, A_{\rho(m)-1}\Big\}$$ of pairwise disjoint anti-amicable signed permutation matrices of order $m$, and equivalently there is $OD\big(m; \ 1_{(\rho(m))}\big)$, where $\rho(m)$ is the Radon-Hurwitz number.
 \end{theorem}
\begin{remark}\label{mut-anti}{\rm
Since the set $A$ in Theorem \ref{setanti} is a set of pairwise anti-amicable matrices, $A_i=-A_i^{\rm T}$ for $1\le i \le \rho(m)-1$, and so $A_i^2=-I_m$ and $A_iA_j=-A_jA_i$ for $1\le i\neq j \le \rho(m)-1$.
}\end{remark}
\begin{theorem}\label{SOD2n}
There is $SOD\ \big(2^n; 1_{(2^n)}, S\big)$ such that $S$ admits a remrep of degree $2^{2^{n-1}-1}$, $n>2$.
\end{theorem}
\begin{proof}
Let $m=2^{2^{n-1}-1}$. It is not hard to see that $\rho(m)=2^n$, for $n>2$. By Theorem \ref{setanti}, there is a set $A=\big\{A_0, A_1, A_2, \ldots, A_{2^n-1}\big\}$ of pairwise disjoint anti-amicable signed permutation matrices of order $m$, where $A_0=I_m$. Let 
\begin{align}\label{signedT}
T=\big<A_1, \ldots, A_{2^n-1}\big>.
\end{align} 
It can be seen that $T$ is a signed subgroup of $SP_m$. Thus, for each $M\in T$, $\overline{M}=M^{-1}=M^{\rm T}$. 
Now let $B=\big\{B_0,B_1,\ldots, B_{2^n-1}\big\}$ be a set of supplementary matrices obtained from all possible $n$-fold Kronecker products of $I$ and $P$, where
$$I= \left[ \begin {array}{cc} 1&0\\ 0&1\end {array} \right]\ {\rm and} \ P= \left[ \begin {array}{cc} 0&1\\ 1&0\end {array} \right]\!.$$
It is easy to see that the matrices in the set $B$ are  pairwise amicable of order $2^n$. Now we show that
$$D=\sum_{i=0}^{2^n-1}A_ix_iB_i,$$ 
is $SOD\ \big(2^n; 1_{(2^n)}, T\big)$, where the $x_i$'s are commuting variables. Note that the $A_ix_i$ is treated as a scalar multiplied on every entry in the $B_i$. We have
\begin{align*}
DD^*&=D\overline{D}^{\rm T}=\bigg(\sum_{i=0}^{2^n-1}A_ix_iB_i\bigg)\bigg(\sum_{i=0}^{2^n-1}\overline{A_i}x_iB_i^{\rm T}\bigg)\\&= 
\sum_{i=0}^{2^n-1}(A_ix_iB_i)(\overline{A_i}x_iB_i^{\rm T}) + \sum_{i=0}^{2^n-1}\sum_{j=i+1}^{2^n-1}\Big((A_ix_iB_i)(\overline{A_j}x_jB_j^{\rm T}) + (A_jx_jB_j)(\overline{A_i}x_iB_i^{\rm T})\Big)\\&=\sum_{i=0}^{2^n-1}A_i\overline{A_i}x_i^2B_iB_i^{\rm T}+
\sum_{i=0}^{2^n-1}\sum_{j=i+1}^{2^n-1}(A_i\overline{A_j}x_ix_jB_iB_j^{\rm T}+ A_j\overline{A_i}x_jx_iB_jB_i^{\rm T})\\&=\sum_{i=0}^{2^n-1}I_mx_i^2I_{2^n}+
\sum_{i=0}^{2^n-1}\sum_{j=i+1}^{2^n-1}(A_i\overline{A_j}x_ix_jB_iB_j^{\rm T}- A_i\overline{A_j}x_ix_jB_iB_j^{\rm T})\\&= \Big(\sum_{i=0}^{2^n-1}x_i^2\Big)I_{2^n}.
\end{align*}
Therefore, $D$ is $SOD\ \big(2^n; 1_{(2^n)}, T\big)$. One may choose the identity map $\pi$ from $T$ to $T$ which is clearly a remrep of degree $m$, and so $D$ is an SOD over  $T$ admitting the remrep $\pi$ of degree $m$.
\end{proof}

\begin{example}{\rm
We show that there is $SOD\ \big(2^3; 1_{(2^3)}, S\big)$ such that $S$ admits a remrep of degree $2^3$. Let 
\begin{align*}
B_0&=I\otimes I\otimes I, \ \ \ \ \ \ \ \ B_4=I\otimes P\otimes P, \\ B_1&=I\otimes I\otimes P , \ \ \ \ \ \ \  B_5=P\otimes I\otimes P, \\ B_2&=I\otimes P\otimes I, \ \ \ \ \ \ B_6=P\otimes P\otimes I, \\ B_3&=P\otimes I\otimes I, 
 \ \  \ \ \ \ B_7=P\otimes P\otimes P.
\end{align*}
Let $S$ be the signed group in \eqref{signedT} with $n=3$. It can be verified that $D=\sum_{i=0}^7 A_ix_iB_i $ is the desired SOD, where the $x_i$'s are commuting variables. Note that $S$ admits a remrep of degree $8$ that is related to the existence of full ODs of type $(1_{(8)})$. From Theorem \ref{setanti}, for $m>8$, there cannot exist $m$ pairwise disjoint anti-amicable signed permutation matrices of order $m$.}
\end{example}

\begin{remark}{\rm 
In the proof of Theorem \ref{SOD2n}, noting Remark \ref{mut-anti}, one may let
\begin{align}\label{signedS}
S=\big<s_1, \ldots, s_{2^n-1}\colon\, s_{\alpha}^2=-1,  s_{\alpha}s_{\beta}=-s_{\beta}s_{\alpha}, 1\le \alpha\neq \beta\le 2^n-1\big>,
\end{align}
and define a map $\phi: \ S\longrightarrow T$ by $\phi(s_{\alpha})=A_{\alpha}$ for $1\le \alpha\le 2^n-1$ such that multiplication is preserved \cite[Section 2.3]{OB}. Therefore, 
$D=\sum_{i=0}^{2^n-1}s_ix_iB_i$, where $s_0=1_S$, is also $SOD\ \big(2^n; 1_{(2^n)}, S\big)$ such that $S$ admits the remrep $\phi$ of degree $m$. 
}\end{remark}

\begin{remark}{\rm
Comparing the maximum number of variables in an OD with the number of variables in the SOD constructed in Theorem \ref{SOD2n}, one can observe that the maximum number of variables in an SOD over a signed group depends on the type of signed group. Determination of the maximum number of variables in an SOD over different signed groups is an interesting and challenging problem. 
}\end{remark}

Craigen \cite{Craigen} showed that if there is $SH(n,S)$ such that $S$ admits a remrep of degree $m$, then there is a Hadamard matrix of order $mn$, where $m$ is the order of a Hadamard matrix. Following his lead, Ghaderpour \cite{Ghader, Ghaderpour} showed the following theorem.

\begin{theorem}{\rm \cite[Theorem 6.26]{Ghaderpour}.}\label{sodtood}
If there is $SOD\ \big(n; \ u_1, \ldots, u_k, S\big)$ such that $S$ admits a remrep $\pi$ of degree $m$,  then there is $OD\ \big(mn; \ mu_1, \dots, mu_k\big)$, where $m$ is the order of a Hadamard matrix.  
\end{theorem}

Using the remrep from $S_{\mathbb C}$ to $SP_2$ defined in \cite{Ghader, Ghaderpour}, we have the following corollary. 
\begin{corollary}\label{CODtoOD}
If there exists $COD\ \big(n; \ u_1, \ldots, u_k\big)$, then there exists $OD\ \big(2n; \ 2u_1, \ldots, 2u_k\big)$.
\end{corollary}

Craigen, Holzmann and Kharaghani in \cite{CHK} showed that if $g_1$ and $g_2$ are complex Golay numbers and $g$ is an even Golay number, then $gg_1g_2$ is a complex Golay number. Using this fact, they showed the following theorem.
\begin{theorem}\label{CGN2}
All numbers of the form $m=2^{a+u}3^b5^c11^d13^e$ are complex Golay numbers, where $a,b,c,d,e$ and $u$ are non-negative integers such that $b+c+d+e\le a+2u+1$ and $u\le c+e.$
\end{theorem}

Following similar techniques to \cite{Kharaghani}, we show the following theorem.
\begin{theorem}\label{s2}
Suppose that $r$ is a Golay number, and $k_1, k_2, \ldots, k_{2^{n-3}-1}$ are complex Golay numbers, where $n>2$. If $m=2\sum_{j=1}^{2^{n-3}-1}k_j+r+1$, then there is 
$$COD\ \Big(2^qm; \ 2^q, 2^qr, 2^{q+1}k_1, \ldots, 2^{q+1}k_{2^{n-3}-1}\Big),$$ where $q=2^{n-1}+n-1$.
\end{theorem}

\begin{proof}
Let $n$ be a positive integer greater than $2$. Suppose that $H$ is a Hadamard matrix of order $2^{n-2}$, $(A;B)$ is a Golay pair of length $r$, and $\big(C^{(j)};D^{(j)}\big)$ is a complex 
Golay pair of length $k_j$ $\big(1\le j\le 2^{n-3}-1\big)$. Let $m=2\sum_{j=1}^{2^{n-3}-1}k_j+r+1$. Consider the following two symbolic arrays:
\begin{align*}
E&=\Big(y,x_1C^{(1)},\ldots, x_{2^{n-3}-1}C^{(2^{n-3}-1)},zA,x_{2^{n-3}-1}C^{(2^{n-3}-1)}_{\overline{R}}, \ldots, x_1C^{(1)}_{\overline{R}}\Big), \\
F&=\Big(y,x_1D^{(1)},\ldots, x_{2^{n-3}-1}D^{(2^{n-3}-1)},zB,x_{2^{n-3}-1}D^{(2^{n-3}-1)}_{\overline{R}}, \ldots, x_1D^{(1)}_{\overline{R}}\Big),
\end{align*}
where the $x_j$'s, $y$ and $z$ are commuting variables. 
Let $e$ be the $2^{n-2}$-dimensional column vector of ones. At this point, the sequences $A$, $B$, $C^{(j)}$'s and $D^{(j)}$'s are treated as scalars, so $E$ and $F$ can be seen as row vectors of dimension $2^{n-2}$, and so $eE$ and $eF$ are square matrices of order $2^{n-2}$ whose entries are these sequences multiplied by the variables. Let $\odot$ denotes entrywise multiplication. For each $j$, $1\le j\le 2^{n-2}$, let $E_j$ and $F_j$  be the circulant matrices of order $m$ whose first rows are the expanded $j$-th rows of $eE\odot H$ and $eF\odot H$,  respectively. In other words, the rows of $eE\odot H$ and $eF\odot H$ have a similar form as the arrays $E$ and $F$ in which expanding the sequences $A$, $B$, $C^{(j)}$'s and $D^{(j)}$'s results in row vectors of dimension $m$.  It can be verified (see \cite{Kharaghani}) that 
\begin{align}\label{Wilcom}
\sum_{j=1}^{2^{n-2}}\big(E_jE_j^*+F_jF_j^*\big)=2^{n-1}\bigg(y^2+rz^2+2\sum_{j=1}^{2^{n-3}-1}k_jx_j^2\bigg)I_m.
\end{align}
For each $j$, $1\le j\le 2^{n-2}$, let 
$$E'_j=\dfrac{1}{2}\big(E_j+E_j^*\big), \ \ E''_j=\dfrac{i}{2}\big(E_j-E_j^*\big), \ \ F'_j=\dfrac{1}{2}\big(F_j+F_j^*\big), \ \ F''_j=\dfrac{i}{2}\big(F_j-F_j^*\big).$$
Note that the coefficients of elements of the $E'_j$'s, $E''_j$'s, $F'_j$'s and $F''_j$'s are in $\{0,\pm 1, \pm i\}$ because of the form of the arrays $E$ and $F$. Now it can be seen that the set 
\begin{align*}
\Omega=\Big\{E'_j-E''_j, E'_j+E''_j, F'_j-F''_j, F'_j+F''_j; \ \ 1\le j\le 2^{n-2}\Big\}
\end{align*}
consists of $2^n$ Hermitian circulant matrices. Moreover,

\begin{align*}
\sum_{j=1}^{2^{n-2}}&\Big(\big(E'_j-E''_j\big)\big(E'_j-E''_j\big)^*+\big(E'_j+E''_j\big)\big(E'_j+E''_j\big)^*\\ &+\big(F'_j-F''_j\big)\big(F'_j-F''_j\big)^*+\big(F'_j+F''_j\big)\big(F'_j+F''_j\big)^*\Big)\\&=
2\sum_{j=1}^{2^{n-2}}\Big(E'^{2}_j+E''^{2}_j+F'^{2}_j+F''^{2}_j\Big)\\ &=
\dfrac{1}{2}\sum_{j=1}^{2^{n-2}}\Big(\big(E_j+E_j^*\big)^2-\big(E_j-E_j^*\big)^2+\big(F_j+F_j^*\big)^2-\big(F_j-F_j^*\big)^2\Big) \\ &=
2\sum_{j=1}^{2^{n-2}}\big(E_jE_j^*+F_jF_j^*\big) \ \ \ \ \ \ \ \ \ \ {\rm from} \ \eqref{Wilcom}\\ &=
2^n\bigg(y^2+rz^2+2\sum_{j=1}^{2^{n-3}-1}k_jx_j^2\bigg)I_m.
\end{align*}
From Theorem \ref{SOD2n}, there is $SOD\ \big(2^n; 1_{(2^n)}, S\big)$ such that $S$ admits a remrep of degree $2^{2^{n-1}-1}$. By Theorem \ref{sodtood}, there is $OD\ \Big(2^q; \ 2^{2^{n-1}-1}_{(2^n)}\Big)$, where $q=2^{n-1}+n-1$. Replacing variables in this OD by the Hermitian circulant matrices in the set $\Omega$, one obtains the desired COD.
\end{proof}
\begin{example}{\rm
Using Theorem \ref{s2}, we show that there is $$COD\ \big(2^{11}\cdot 31; \ 2^{11}\cdot 1, 2^{11}\cdot 8, 2^{11}\cdot 22\big).$$ 
Let $e$ be the 4-dimensional column vector of all ones, $(A;B)$ be a Golay pair of length $8$, and $(C;D)$ be a complex Golay pair of length $11$ as follows:
\begin{align*}
A&=(1,1,1,-,1,1,-,1), \ \ \ \ \ \ \ \ \ B=(1,1,1,-,-,-,1,-), \\
C&=(1,i,-,1,-,i,\overline{i},-,i,i,1), \ \ D=(1,1,\overline{i},\overline{i},\overline{i},1,1,i,-,1,-).
\end{align*}
Let $E=\big(y,xC,zA,xC_{\overline{R}}\big)$, $ F=\big(y,xD,zB,xD_{\overline{R}}\big)$ and
$$H=\left[\begin {array}{cccc} 1&1&1&1\\ 1&-&1&- \\ 1&1&-&-\\ 1&-&-&1 \end {array} \right]\!.$$  
We have 
$${\footnotesize eE\odot H=\left[\begin {array}{rrrr} y&xC&zA&xC_{\overline{R}}\\ y&-xC&zA&-xC_{\overline{R}} \\
 y&xC&-zA&-xC_{\overline{R}}\\ y&-xC&-zA&xC_{\overline{R}} \end {array} \right] \ {\rm and}\ eF\odot H=\left[\begin {array}{rrrr} y&xD&zB&xD_{\overline{R}}\\ y&-xD&zB&-xD_{\overline{R}} \\
 y&xD&-zB&-xD_{\overline{R}}\\ y&-xD&-zB&xD_{\overline{R}} \end {array} \right]\!.}$$
Let 
\begin{align*}
E_1&={\rm circ} \ \big(y,xC,zA,xC_{\overline{R}}\big),  \ \ \ \ \ \ \ \ F_1={\rm circ} \ \big(y,xD,zB,xD_{\overline{R}}\big), \\
E_2&={\rm circ} \ \big(y,-xC,zA,-xC_{\overline{R}}\big), \ \ \ F_2={\rm circ} \ \big(y,-xD,zB,-xD_{\overline{R}}\big), \\
E_3&={\rm circ} \ \big(y,xC,-zA,-xC_{\overline{R}}\big),  \ \ \ F_3={\rm circ} \ \big(y,xD,-zB,-xD_{\overline{R}}\big), \\
E_4&={\rm circ} \ \big(y,-xC,-zA,xC_{\overline{R}}\big),  \ \ \ F_4={\rm circ} \ \big(y,-xD,-zB,xD_{\overline{R}}\big).
\end{align*}
Note that the first rows of the circulant matrices above have dimension $31$, and we wrote them symbolically because of space limitations. From each of the circulant matrices above, one obtains two Hermitian circulant matrices. As an example, $E_3$ is the circulant matrix with the following first row:
$$\big(y,x,ix,\underline{x},x,\underline{x},ix,\overline{i}x,\underline{x},ix,ix,x,\underline{z},\underline{z},\underline{z},z,\underline{z},\underline{z},z,\underline{z},\underline{x},ix,ix,x,\overline{i}x,ix,x,\underline{x},x,ix,\underline{x}\big),$$
where $\underline{u}$ means $-u$. The following rows are the first rows of the supplementary Hermitian circulant matrices $E'_3=\dfrac{1}{2}(E_3+E_3^*)$ and $E''_3=\dfrac{i}{2}(E_3-E_3^*)$, respectively:
$$\big(y,0_{(11)},\underline{z},0,\underline{z},0,0,\underline{z},0,\underline{z},0_{(11)}\big),$$
$$ \big(0,ix,\underline{x},\overline{i}x,ix,\overline{i}x,\underline{x},x,\overline{i}x,\underline{x},\underline{x},ix,0,\overline{i}z,0,iz,
\overline{i}z,0,iz,0,\overline{i}x,\underline{x},\underline{x},ix,x,\underline{x},ix,\overline{i}x,ix,\underline{x},\overline{i}x\big).$$
Therefore, $E'_3+E''_3$ and $E'_3-E''_3$ are the desired two Hermitian circulant matrices obtained from $E_3$. Continuing this process, one obtains $16$ complementary Hermitian circulant matrices of order $31$. Replacing these matrices with variables in  $OD\ \big(2^{11}; 2^7_{(16)}\big)$ 
obtained from Theorem \ref{SOD2n} and \ref{sodtood}, one finds
\begin{align*}
COD\ \big(2^{11}\cdot 31; \ 2^{11}\cdot 1, 2^{11}\cdot 8, 2^{11}\cdot 22\big).
\end{align*}
}
\end{example}
\begin{remark}{\rm
%{\rm Replacing the $16$ complementary Hermitian circulant matrices in the above example with variables in $OD\ \big(2^{10}; \ 2^6_{(16)}\big)$ obtained 
%in Theorem \ref{128orthogonal}, we get a $$COD\ \big(2^{10}\cdot 31; \ 2^{10}\cdot 1, 2^{10}\cdot 8, 2^{10}\cdot 22\big).$$ }
The method of constructing the COD above and the infinite family of CODs in Theorem \ref{s2} is interesting because it uses (complex) Golay pairs and circulant matrices. Moreover, these CODs have no zero entries.  Applying Corollary \ref{CODtoOD} to the COD above, one obtains $$OD\ \big(2^{12}\cdot 31; \ 2^{12}\cdot 1, 2^{12}\cdot 8, 2^{12}\cdot 22\big).$$}
\end{remark}

\section{Other constructions for ODs using SODs}
Suppose that $m=2^{2^{n-1}-1}$ for some $n>2$. By Theorem \ref{setanti}, there is a set $A=\big\{I_m, A_1, A_2, \ldots, A_{2^n-1}\big\}$ of pairwise disjoint anti-amicable signed permutation matrices of order $m$. If we let $A'=A_{2^n-3}A_{2^n-2}A_{2^n-1}$ and $T'=\big<A', A_1, \ldots, A_{2^n-4}\big>$, then clearly $T'\le T\le SP_m$, where $T$ is the signed group in \eqref{signedT}. Let 
\begin{align}\label{signedSp} \nonumber
S'=\Big<s,s_1, \ldots, &s_{2^n-4}\colon\, s^2=1, \ s_{\alpha}^2=-1, \\   &ss_{\alpha}=-s_{\alpha}s, \ s_{\alpha}s_{\beta}=-s_{\beta}s_{\alpha}, \ \ 1\le \alpha\neq \beta\le 2^n-4\Big>.
\end{align}
It is easy to see that $S'\le S$, where $S$ is the signed group in \eqref{signedS}. We define a map $\pi: \ S'\longrightarrow T'$ by $\pi(s)=A'$ and $\pi(s_{\alpha})=A_{\alpha}$ for $1\le \alpha\le \ 2^n-4$ such that multiplication is preserved. 

In the following example, we show how one can construct a full SOD over $S'$ admitting the remrep $\pi$, which leads to a full OD.

\begin{example}{\rm
In order to construct a full SOD of type $(1,1,1,9,9,11)$, we may first construct a full SOD of type $(1_{(8)}, 8,8,8)$. To do so, first let $n=4$ in signed group $S'$ in \eqref{signedSp}
that admits a remrep of degree $2^7$, and let 
$$A= \left[\begin {array}{cc} s_1&1\\ 1&s_1\end {array} \right], \ \ I_s=\left[\begin {array}{cc} s&0\\ 0&s\end {array} \right], \ \ P_s=\left[ \begin {array}{cc} 0&s\\ s&0\end {array} \right],$$ where $s,s_1\in S'$, $s^2=1$, $s_1^2=-1$ and $ss_1=-s_1s$. It can be seen that $A, I_s$ and $P_s$ are pairwise amicable. Also, let
\begin{align*}
&B_1=I\otimes I \otimes I_s\otimes I_s\otimes I_s, \ \ \ \ \ \ \  B_9=P\otimes I\otimes A\otimes A\otimes A,         \\
&B_2=I\otimes I\otimes I_s\otimes I_s\otimes P_s, \ \ \ \ \ \     B_{10}=I\otimes P\otimes A\otimes A\otimes A,       \\ 
&B_3=I\otimes I\otimes I_s\otimes P_s\otimes I_s, \ \ \ \ \ \    B_{11}= P\otimes P\otimes A\otimes A\otimes A.        \\  
&B_4=I\otimes I\otimes P_s\otimes I_s\otimes I_s,  \\
&B_5=I\otimes I\otimes I_s\otimes P_s\otimes P_s,  \\
&B_6=I\otimes I\otimes P_s\otimes I_s\otimes P_s,  \\
&B_7=I\otimes I\otimes P_s\otimes P_s\otimes I_s, \\
&B_8=I\otimes I\otimes P_s\otimes P_s\otimes P_s,
\end{align*}
It can be easily verified that the $B_i$'s are supplementary pairwise amicable matrices. Using the relationships $s=\overline{s}$, $s_1s=\overline{s_1s}$, $ss_{\alpha}s_{\beta}=s_{\alpha}s_{\beta}s$, $s_1s_{\alpha}s_{\beta}=s_{\alpha}s_{\beta}s_1$ and $s_{\alpha}s_{\beta}=-s_{\beta}s_{\alpha}$ for $2\le \alpha\neq \beta\le 12$, 
it can be verified that $\sum_{i=1}^{11}B_i s_{i+1}x_i$ 
is $SOD\ \big(2^5; \ 1_{(8)}, 8,8,8, S'\big)$ such that $S'$ admits the remrep of degree $2^7$, where $x_i$'s are commuting variables. 
Equating variables in this SOD over $S'$, one obtains $SOD\ \big(2^5; \ 1,1,1, 9,9,11, S'\big)$ such that $S'$ admits the remrep of degree $2^7$. From Theorem \ref{sodtood}, one also obtains 
$$OD\ \big(2^7\cdot 2^5;\  2^7\cdot 1_{(3)},2^7\cdot 9_{(2)},2^7\cdot 11\big).$$}
\end{example}

\section{Discussion}
Hadamard matrices, ODs and CODs have many applications in coding theory, cryptography, signal processing, wireless networking and communications \cite{Seb, Seberry, Tarokh}. We observed that Hadamard matrices, ODs, CODs, SWs and SHs are specific SODs. An SOD over a signed group other than $S_{\mathbb R}$ or $S_{\mathbb C}$ may also have applications in the areas mentioned above. As we showed in Section 2, there exist types and orders in which there are no SODs of those types and orders. We also constructed some interesting families of full SODs. Using some special signed groups and extensive algebra may result in constructing other interesting families of full SODs and consequently full ODs and Hadamard matrices.

\section{Acknowledgement}
Sections 2 and 3 of this paper are a part of the author's Ph.D. thesis written under the direction of Professor Hadi Kharaghani at the University of Lethbridge. The author would like to thank Professors Hadi Kharaghani and Rob Craigen for their time and great help. The author also thank the reviewers of the journal for their time and essential comments that improved significantly the presentation of the paper.

\end{document}